\documentclass{article}
\usepackage{amssymb,amsfonts,amsmath,amsthm,amscd}
\usepackage{rawfonts}
\usepackage{pictexwd}
\usepackage{makeidx}
\usepackage[english]{babel}
\makeindex

\theoremstyle{thm}
\newtheorem{thm}{Theorem}[section]

\theoremstyle{lem}

\newtheorem{prop}[thm]{Proposition}
\newtheorem{defn}[thm]{Definition}
\theoremstyle{rem}

\newcommand{\Hh}{\mathcal{H}}
\newcommand{\G}{\mathcal{G}}

\title{\textbf{On normal subgroupoids}}
\author{Jes\'{u}s \'{A}vila and V\'{\i}ctor Mar\'{\i}n$^{\text{1}}$\\
   \small $^{\text{1}}$Departamento de Matem\'{a}ticas y Estad\'{i}stica \\
   \small Universidad del Tolima\\
   \small Santa Helena, Ibagu\'{e}, Colombia\\
   \small e-mail: javila@ut.edu.co, vemarinc@ut.edu.co}
   \date{\today}
\begin{document}
\maketitle
\begin{abstract}
\noindent
In this paper we present some algebraic properties of subgroupoids and normal subgroupoids. We define the normalizer of a wide subgroupoid $\Hh$ and show that, as in the case of groups, the normalizer is the greatest wide subgroupoid of the groupoid $\G$ in which $\Hh$ is normal. Furthermore, we give the definition of center and commutator and prove that both are normal subgroupoids, the first one of the union of all the isotropy groups of $\G$ and the second one of $\G$. Finally, we introduce the concept of inner isomorphism of $\mathcal{G}$ and show that the set of all the inner isomorphisms of $\G$ is a normal subgroupoid, which is isomorphic to the quotient groupoid of $\mathcal{G}$ by its center $\mathcal{Z}(\mathcal{G})$, which extends to groupoids a well-known result in groups.
\end{abstract}


\section{Introduction}
The notion of groupoid (denominated Brandt groupoid) was first introduced in \cite{B}, from an algebraic point of view. Next, Brandt groupoids were generalized by C. Ehresmann in \cite{E}, where he added other structures such as topological and differentiable.
Other equivalent definition of groupoid and its properties appear in \cite{Br}, where a groupoid is defined as a small category where each morphism is invertible.

In \cite[Definition 1.1]{I} the author follows the definition given by Ehresmann  and presents the notion of groupoid as a particular case of an universal algebra,  he defines strong homomorphism for groupoids and proves the correspondence theorem in this context. The Cayley Theorem for groupoids is also presented in \cite[Theorem 3.1]{I2}.

The definition of groupoid from an approach axiomatic, such as that of group, is presented in \cite{L}. In that sense,  Paques and Tamusiunas give necessary and sufficient conditions to be a normal subgroupoid and they construct the quotient groupoid \cite{PT}.  In \cite{AMP}, the isomorphism theorems are proved and one application of these to the normal series is presented.

The purpose of this paper is to introduce several concepts into groupoids, which are analogous to those defined in groups  such as center, normalizer, commutator and inner isomorphism. In addition, the normality of these subgroupoids and other properties that they satisfy are studied. The paper is organized of the following manner. In Section 2 we present some preliminaries and basic results on groupoids and subgroupoids, which are used in the following sections. Next, in Section 3 we present some algebraic properties of normal subgroupoids and introduce the normalizer of a subgroupoid $\mathcal{H}$, $\mathcal{N}_{\mathcal{G}}(\mathcal{H})$. Furthermore, we show that the normalizer is the greatest wide subgroupoid of the groupoid $\mathcal{G}$ in which $\mathcal{H}$ is normal (Proposition \ref{normalizer}). In Section 4 we introduce the center $\mathcal{Z}(\mathcal{G})$ and the commutator subgroupoid $\mathcal{G}'$ of $\mathcal{G}$ and prove that they are normal subgroupoids and that $\mathcal{G}/\mathcal{G}'$ is the largest abelian quotient of $\mathcal{G}$ (Propositions \ref{center},\, \ref{commutator}). Finally, in Section 5 we define inner isomorphisms of a groupoid and prove that the set of all the inner isomorphisms of $\mathcal{G}$, $\mathcal{I}(\mathcal{G})$, is a normal subgroupoid and it is isomorphic to the quotient groupoid $\mathcal{G}/\mathcal{Z}(\mathcal{G})$ (Proposition \ref{inner}).

\medskip

\section{Preliminaries and basic results}

Now, we give the definition  of groupoid from purely algebraic point of view. We follow the definition presented in \cite{L}.


\begin{defn}\cite[p. 78]{L}.\label{d2}
Let $\mathcal{G}$ be a set equipped with a partial binary operation which is denoted by concatenation. If $g,h \in \mathcal{G}$ and the product $gh$ is defined, we write $\exists gh$. A element $e\in \mathcal{G}$  is called an identity if $\exists eg$ implies $eg=g$ and $\exists g'e$ implies $g'e=g'$. The set of identities of $\mathcal{G}$ is denoted $\mathcal{G}_0$. $\mathcal{G}$ is said to be a groupoid if the following axioms hold:
\begin{enumerate}
\item $\exists g(hl)$ if, and only if, $\exists (gh)l$, and $g(hl)=(gh)l$.
\item $\exists g(hl)$ if, and only if, $\exists gh$ and $\exists hl$.
\item For each $g\in \mathcal{G}$, there exist unique elements $d(g), r(g)\in \mathcal{G}$ such that $\exists gd(g)$ and $\exists r(g)g$ and $gd(g)=g=r(g)g$.
\item For each $g\in \mathcal{G}$, there exist a element $g^{-1}\in \mathcal{G}$ such that $d(g)=g^{-1}g$ and $r(g)=gg^{-1}$.
\end{enumerate}
\end{defn}



In groupoids is important to characterize in which case exist the product of two elements. It can be proved that if $x,y\in \mathcal{G}$ then $\exists xy$ if and only if $d(g)=r(h)$ \cite[Lemma 2.3]{AMP}. The following proposition shows several important properties that are fulfilled in the groupoids.




\begin{prop}\cite[Proposition 2.7]{AMP}.\label{t1}
Let $\mathcal{G}$ be a groupoid. Then for each $g,h, k, l\in G$ we have
\begin{enumerate}
\item The element $g^{-1}$  is unique and $(g^{-1})^{-1}=g$.
\item If $\exists (gh)(kl)$, then $(gh)(kl)=g[(hk)l]$.
\item $d(gh)=d(h)$ and $r(gh)=r(g)$.
\item $\exists gh$ if, and only if, $\exists h^{-1}g^{-1}$ and, in this case, $(gh)^{-1}=h^{-1}g^{-1}$.
\end{enumerate}
\end{prop}

The following results are obtained easily from the previous proposition.

\begin{prop}\cite[Proposition 2.8]{AMP}.\label{propiedades de d y r}
If $\mathcal{G}$ is a groupoid and $g\in G$ then $d(g)=r(g^{-1}), d(d(g))=d(g)=r(d(g))$ and $d(r(g))=r(g)=r(r(g))$.
\end{prop}





If $\mathcal{G}$ is a groupoid, then the identities of $\mathcal{G}$ are the elements $e=d(g)$ with $g\in \mathcal{G}$ \cite[Proposition 2.10]{AMP} and we set $\mathcal{G}_0=\{e=d(g)\mid g\in \mathcal{G}\}$. Now, if $e\in \mathcal{G}_0$ then by
Proposition \ref{propiedades de d y r} it has that $d(e)=r(e)=e$,
  $\exists ee$ and $ee=e$ and $e^{-1}=e$ and moreover the set $\G_e=\{g\in \mathcal{G}\mid d(g)=r(g)=e\}$ is a group with identity element $e$, which is called the isotropy group associated to $e$.

  \vspace{.3cm}

  These isotropy groups are very important, because they allow us to extend some concepts from groups to groupoids. For example, a groupoid $\mathcal{G}$ is called abelian if all its isotropy groups are abelian \cite[Definition 1.1]{MA}. In this way the set ${\rm Iso}(\G)=\bigcup_{e\in \G_0}\G_e$, which is called the isotropy subgroupoid of $\mathcal{G}$, is essential for the study of the groupoids.

\vspace{.2cm}

Now, we present the definition of subgroupoid and wide subgroupoid and prove some algebraic properties of these substructures, which are also valid in groups.

\begin{defn}\cite[p. 107]{PT}.
Let $\mathcal{G}$ be a groupoid and $\mathcal{H}$ a nonempty subset of $\mathcal{G}$.  $\mathcal{H}$ is said to be a subgroupoid of $\mathcal{G}$ if for all $g,h \in \mathcal{H}$ it satisfies:
\begin{enumerate}
\item $g^{-1}\in \mathcal{H}$;
\item $\exists gh \Rightarrow gh\in \mathcal{H}$.
\end{enumerate}
In this case we denote $\mathcal{H}< \mathcal{G}$. In addition, if $\mathcal{H}_0=\mathcal{G}_0$ (or equivalently $\mathcal{G}_0\subseteq \mathcal{H}$) then $\mathcal{H}$ is called a wide subgroupoid of $\mathcal{G}$.
\end{defn}


Note that if $\mathcal{G}$ is a groupoid then the sets $\{d(g)\}$ ($g\in \mathcal{G}$), $\mathcal{G}_e$ ($e\in \mathcal{G}_0$), $\mathcal{G}_0$, $Iso(\mathcal{G})$ and $\mathcal{G}$ are subgroupoids of $\mathcal{G}$. Also, it is easy to see that if $\mathcal{H}$ is a subgroupoid then the set $\mathcal{H}\cup \mathcal{G}_0$ is a wide subgroupoid of $\mathcal{G}$.

\vspace{.3cm}

Moreover, if $\mathcal{H}$ is a wide subgroupoid of $\mathcal{G}$ and $g\in \mathcal{G}$ then $g^{-1}\mathcal{H}g=\{g^{-1}hg\mid h\in \mathcal{H}\textnormal{ and } r(h)=d(h)=r(g)\}$ is a subgroupoid of $\mathcal{G}$. In fact,  note that $r(g)\in \mathcal{H}$ and then $\exists g^{-1}r(g)g$ that is $d(g)=g^{-1}r(g)g\in g^{-1}\mathcal{H}g$. If $x,y\in \mathcal{G}$ then $x=g^{-1}hg$ and $y=g^{-1}tg$ with $h,t\in \mathcal{H}$ and $r(h)=d(h)=r(t)=d(t)=r(g)$. Since $d(g^{-1}hg)=d(g)=r(g^{-1})=r(g^{-1}tg)$ then $\exists (g^{-1}hg)(g^{-1}tg)$ and it has that
    $xy = (g^{-1}hg)(g^{-1}tg)
     = (g^{-1}h)(r(g))(tg)
     = (g^{-1}h)(r(t)t)g
     = g^{-1}htg$.
  Now since $\exists ht$ then $ht\in \mathcal{H}$ and thus $xy\in g^{-1}\mathcal{H}g$. Finally, if $x\in \mathcal{G}$ then $x=g^{-1}hg$ with $h\in \mathcal{H}$ and $r(h)=d(h)=r(g)$. Thus, $x^{-1}=g^{-1}h^{-1}g\in g^{-1}\mathcal{H}g$ because $h^{-1}\in \mathcal{H}$. Hence, the set $g^{-1}\mathcal{H}g$ is a subgroupoid of $\mathcal{G}$. Note in particular that $g^{-1}\mathcal{H}g$ is a subgroup of $\mathcal{G}_{d(g)}$ and $g^{-1}\mathcal{H}g=g^{-1}\mathcal{H}_{r(g)}g$.

\begin{prop}\label{interseccion de subgrupoides}
  Let $\mathcal{G}$ be a groupoid and $\left\{\mathcal{H}_i\right\}_{i\in I}$ a family of subgroupoids of $\mathcal{G}$. Then:
  \begin{enumerate}
  \item If $\bigcap _{i\in I}\mathcal{H}_i\neq \emptyset$, then $\bigcap _{i\in I}\mathcal{H}_i$ is a subgroupoid of $\mathcal{G}$.
  \item If $\mathcal{H}_i$ is wide for each $i\in I$ then $\bigcap _{i\in I}\mathcal{H}_i$ is a wide subgroupoid of $\mathcal{G}$.
\end{enumerate}
\end{prop}
\begin{proof}
1. Let $A=\bigcap _{i\in I}\mathcal{H}_i$. Since $A\neq \emptyset$, let $a,b\in A$ and suppose that $\exists ab$. Then $d(a)=r(b)$ and $a,b\in \mathcal{H}_i$ for each $i\in I$. Thus $ab\in \mathcal{H}_i$ for each $i\in I$ that is $ab\in \bigcap _{i\in I}\mathcal{H}_i=A$. Finally, if $a\in A$ then $a\in \mathcal{H}_i$ for each $i\in I$ and then $a^{-1}\in \mathcal{H}_i$ for each $i\in I$. So, $a^{-1}\in \bigcap _{i\in I}\mathcal{H}_i=A$ and the result follows.

2. It is enough to observe that $\emptyset \neq \mathcal{G}_0\subseteq A$ and to apply item (i).
\end{proof}

\begin{prop}\label{SG}
  Let $\mathcal{G}$ be a groupoid and $\emptyset \neq B\subseteq G$. Then there exists the smallest subgroupoid of $\mathcal{G}$ which contains $B$.
\end{prop}
\begin{proof}
Let $\mathfrak{F}=\{\mathcal{H}\subseteq \mathcal{G}\mid \mathcal{H}\textnormal{ is a subgrupoid of }\mathcal{G}\textnormal{ and }B\subseteq \mathcal{H}\}$. Then $\mathfrak{F}\neq \emptyset$ since $\mathcal{G}\in \mathfrak{F}$ and $\bigcap \mathfrak{F}\neq \emptyset$ because $B\subseteq \bigcap \mathfrak{F}$. Thus $\bigcap \mathfrak{F}$ is a subgroupoid of $\mathcal{G}$ by virtue of Proposition \ref{interseccion de subgrupoides}. Moreover, it is clear that $\bigcap \mathfrak{F}$ is the smallest subgroupoid of $\mathcal{G}$ such that $B\subseteq \bigcap \mathfrak{F}$. The other part is clear.
\end{proof}

If $\mathcal{G}$ is a groupoid and $\emptyset \neq B\subseteq \mathcal{G}$, then the subgroupoid given in the previous proposition will be called \textbf{subgroupoid generated by $B$} and it will be denoted by $\langle B\rangle$. It can be proved that the set $\langle B\rangle$ is given by $\langle B\rangle=\left\{x_1^{\alpha_1}x_2^{\alpha_2}...x_n^{\alpha_n}\mid \exists x_1^{\alpha_1}x_2^{\alpha_2}...x_n^{\alpha_n}, x_i\in B, \alpha_i\in \{1,-1\}\, \forall i, n\in \mathbb{N}  \right\}$. Also, note that $\langle B\rangle _w=\langle B\rangle \cup \mathcal{G}_0$ is a wide subgroupoid and it will be called \textbf{wide subgroupoid generated by $B$}.

\vspace{.3cm}

If $\mathcal{G}$ is a groupoid and $\mathcal{H}$, $\mathcal{K}$ are wide subgroupoids of $\mathcal{G}$, we define the set
$$\mathcal{H}\mathcal{K}:=\{hk\mid d(h)=r(k), h\in \mathcal{H}, \, k\in\mathcal{K}\}.$$ Note that $\mathcal{H}\mathcal{K}\neq \emptyset$ since for $g\in \mathcal{G}$, $d(g)\in \mathcal{H}$ and $d(g)\in \mathcal{K}$ and thus $d(g)=d(g)d(g)\in \mathcal{H}\mathcal{K}$. Hence $\mathcal{G}_0\subseteq \mathcal{H}\mathcal{K}$.

\begin{prop}\label{producto HK=KH}
Let $\mathcal{G}$ be a groupoid and $\mathcal{H}$, $\mathcal{K}$ wide subgroupoids of $\mathcal{G}$. Then $\mathcal{H}\mathcal{K}$ is a wide subgroupoid of $\mathcal{G}$ if and only if $\mathcal{H}\mathcal{K}=\mathcal{K}\mathcal{H}$.
\end{prop}
\begin{proof}
  Let $x\in \mathcal{H}\mathcal{K}$. By assumption, $\mathcal{H}\mathcal{K}\leq \mathcal{G}$
  so $x^{-1}\in \mathcal{H}\mathcal{K}$. Then $x^{-1}=hk$ with $h\in \mathcal{H}$, $k\in \mathcal{K}$,
  $d(h)=r(k)$ and hence $x=(x^{-1})^{-1}=(hk)^{-1}=k^{-1}h^{-1}\in \mathcal{K}\mathcal{H}$. On the other
  hand, we know that $G_0\subseteq \mathcal{K}\mathcal{H}$. If $y\in \mathcal{K}\mathcal{H}$ then $y=kh$ with
  $k\in \mathcal{K}$, $h\in \mathcal{H}$, $d(k)=r(h)$. Then $k^{-1}\in \mathcal{K}$ and $h^{-1}\in \mathcal{H}$
  and by assumption $\exists h^{-1}k^{-1}$ and $h^{-1}k^{-1}\in \mathcal{H}\mathcal{K}$. Then $y=kh=(h^{-1}k^{-1})^{-1}\in \mathcal{H}\mathcal{K}$.

  Conversely, suppose that $\mathcal{H}\mathcal{K}=\mathcal{K}\mathcal{H}$. First note that  $\mathcal{H}\mathcal{K}\neq \emptyset$, if $x,y\in \mathcal{H}\mathcal{K}$ then $x=hk$, $h\in \mathcal{H}$, $k\in \mathcal{K}$, $d(h)=r(k)$ and $y=st$, $s\in \mathcal{H}$, $t\in \mathcal{K}$, $d(s)=r(t)$. If $\exists xy$ then $xy=(hk)(st)=h(ks)t=h(s'k')t=(hs')(k't)\in \mathcal{H}\mathcal{K}$. Finally, if $x\in \mathcal{H}\mathcal{K}$
  then $x=hk$, $h\in \mathcal{H}$, $k\in \mathcal{K}$, $d(h)=r(k)$. Thus $x^{-1}=(hk)^{-1}=k^{-1}h^{-1}\in \mathcal{K}\mathcal{H}\subseteq \mathcal{H}\mathcal{K}$. That is, $\mathcal{H}\mathcal{K}\leq \mathcal{G}$ and since $G_0\subseteq \mathcal{H}\mathcal{K}$ we conclude that $\mathcal{H}\mathcal{K}$ is a wide subgroupoid of $\mathcal{G}$.
  \end{proof}

\section{Normal subgroupoids}
In this section we present the definition of normal subgroupoid and present several properties of these classes of subgroupoids. We also introduce the normalizer of a subgroupoid $\mathcal{H}$ and prove some of its algebraic properties. In particular, we prove that the normalizer is the greatest wide subgroupoid of $\mathcal{G}$ in which $\mathcal{H}$ is normal.

\begin{defn}
Let $\mathcal{G}$ be a groupoid. The subgroupoid $\mathcal{H}$ of $\mathcal{G}$
is said to be normal, denoted by $\mathcal{H}\lhd \mathcal{G}$, if $\mathcal{H}$ is wide and $g^{-1}\mathcal{H}g\subseteq \mathcal{H}$ for all $g\in \mathcal{G}$.
\end{defn}

Note that $\mathcal{G}$ is a normal  subgroupoid of $\mathcal{G}$. Now, $\mathcal{G}_0$ is a wide subgroupoid of $\mathcal{G}$ and if $g\in \mathcal{G}$ then $g^{-1}\mathcal{G}_0g=\{d(g)\}\subseteq \mathcal{G}_0$. That is, $\mathcal{G}_0$ is also a normal subgroupoid of $\mathcal{G}$. Moreover, $Iso(\mathcal{G})$ is a normal subgroupoid of $\mathcal{G}$ and if $\mathcal{G}$ is an abelian groupoid and $\mathcal{H}$ is a wide subgroupoid of $\mathcal{G}$ then  $Iso(\mathcal{H})$ is normal in $Iso(\mathcal{G})$.

\vspace{.3cm}

Normality also can be characterized as follows \cite{Br}: the subgroupoid $\mathcal{H}$ is said to be normal if $\mathcal{H}_0=\mathcal{G}_0$ and $g^{-1}\mathcal{H}_{r(g)}g=\mathcal{H}_{d(g)}$ for all $g\in \mathcal{G}$. The equivalence between the two definitions of normality can be consulted in \cite[Lemma 3.1]{PT}.





\vspace{.3cm}
The following proposition extends some well-known results from normal subgroups to normal subgroupoids.

\begin{prop}\label{SIT}
Let $\mathcal{G}$ be a grupoid. Then:
\begin{enumerate}
  \item If $\left\{\mathcal{H}_i\right\}_{i\in I}$ is a family of normal subgroupoids of $\mathcal{G}$, then $\bigcap _{i\in I}H_i$ is a normal subgroupoid of $\mathcal{G}$.
  \item If $\emptyset \neq B\subseteq \mathcal{G}$, then there exist the smallest normal subgroupoid of $\mathcal{G}$ which contains $B$.
  \item If $\mathcal{H}$ is a subgroupoid of $\mathcal{G}$ and $\mathcal{K}$ is a normal subgroupoid of $\mathcal{G}$ such that $d(k)=r(k)$ for all $k\in \mathcal{K}$, then $\mathcal{H}\mathcal{K}$ is a subgroupoid of $\mathcal{G}$.
  \item If $\mathcal{H}$ and $\mathcal{K}$ are normal subgroupoids of $\mathcal{G}$ with $d(k)=r(k)$ for all $k\in \mathcal{K}$, then $\mathcal{H}\mathcal{K}$ is a normal subgroupoid of $\mathcal{G}$.
  \item If $\mathcal{H}$ is a wide subgroupoid of $\mathcal{G}$ and $\mathcal{K}$ is a normal subgroupoid of $\mathcal{G}$, then $\mathcal{H}\cap \mathcal{K}$ is a normal subgroupoid of $\mathcal{H}$.
  \item If $\mathcal{H}$ and $\mathcal{K}$ are normal subgroupoids of $\mathcal{G}$ such that $\mathcal{H}\cap \mathcal{K}=\mathcal{G}_0$, then $hk=kh$ for all $h\in \mathcal{H}$ and $k\in \mathcal{K}$ such that $r(h)=d(h)=r(k)=d(k)$.
\end{enumerate}
\end{prop}
\begin{proof}
  1. By Proposition \ref{interseccion de subgrupoides} (item 2) we have that $\bigcap _{i\in I}H_i$ is a wide subgrupoid of $\mathcal{G}$. Now, let $g\in \mathcal{G}$ and $h\in \bigcap _{i\in I}H_i$ such that $r(h)=d(h)=r(g)$. Then $\exists g^{-1}hg$ and since $h\in H_i$ for each $i\in I$, we have that $g^{-1}hg\in H_i$ for each $i\in I$. That is, $g^{-1}hg\in \bigcap _{i\in I}H_i$.

     2. It is enough to take the collection of normal subgroupoids of $\mathcal{G}$ that contains $B$ and to apply previous item.

    3. First note that $\mathcal{H} \mathcal{K}\neq \emptyset$ since if $h\in \mathcal{H}$ then $h=hd(h)\in \mathcal{H}\mathcal{K}$. Let $x,y\in \mathcal{H} \mathcal{K}$ then there exist $h,s\in \mathcal{H}$ and $k,t\in \mathcal{K}$ with $d(h)=r(k)$ and $d(s)=r(t)$ such that $x=hk$ and $y=st$. If $\exists xy$ then $d(k)=r(s)$ and $xy=(hk)(st)=hkst=hr(k)kst$ and since $r(k)=d(k)=r(s)$ we have that $\exists hr(s)kst=hss^{-1}kst=(hs)(s^{-1}ks)t\in \mathcal{H}\mathcal{K}$. Finally, if $x=hk$ with $h\in \mathcal{H}$, $k\in \mathcal{K}$ and $d(h)=r(k)$ then $x^{-1}=k^{-1}h^{-1}=d(k)k^{-1}h^{-1}$ and since $d(k)=r(k)=d(h)$ we have that $\exists d(h)k^{-1}h^{-1}$ and thus $x^{-1}=d(h)k^{-1}h^{-1}=h^{-1}hk^{-1}h^{-1}=h^{-1}(hk^{-1}h^{-1})\in \mathcal{H}\mathcal{K}$ and the result follows.

    4. By previous item $\mathcal{H}\mathcal{K}$ is a subgroupoid of $\mathcal{G}$. Let $g\in \mathcal{G}$, $h\in \mathcal{H}$ and $k\in \mathcal{K}$ with $r(h)=d(k)=r(g)$ then $\exists g^{-1}hkg$ and thus $d(h)=r(k)$. Then $g^{-1}hkg=g^{-1}hr(k)kg=g^{-1}hr(g)kg=(g^{-1}hg)(g^{-1}kg)\in \mathcal{H}\mathcal{K}$. That is, $\mathcal{H}\mathcal{K}$ is a normal subgroupoid of $\mathcal{G}$.

    5. It is clear that $\mathcal{H}\cap \mathcal{K}$ is a wide subgroupoid of $\mathcal{H}$. Let $g\in \mathcal{H}$ and $h\in \mathcal{H}\cap \mathcal{K}$ with $r(h)=d(h)=r(g)$, then $\exists g^{-1}hg$ and by assumptions it follows that $g^{-1}hg\in \mathcal{H}\cap \mathcal{K}$.

    6. First note that each $d(g)$ satisfies the assumptions of proposition. If $h\in \mathcal{H}$ and $k\in \mathcal{K}$ with $r(h)=d(h)=r(k)=d(k)$, then $\exists h^{-1}k^{-1}hk$. And we obtain that $h^{-1}k^{-1}hk=h^{-1}(k^{-1}hk)\in \mathcal{H}$ and $h^{-1}k^{-1}hk=(h^{-1}k^{-1}h)k\in \mathcal{K}$. Thus, $h^{-1}k^{-1}hk\in \mathcal{G}_0$, that is, $h^{-1}k^{-1}hk=d(g)$ for some $g\in \mathcal{G}$. Then, $d(k)=d(h^{-1}k^{-1}hk)=d(g)$ and thus $r(h)=d(h)=r(k)=d(k)=d(g)$. Hence
  \begin{align*}
    h^{-1}k^{-1}hk &=d(g) \\
    h(h^{-1}k^{-1}hk) &=hd(h) \\
    r(h)(k^{-1}hk) &= h \\
    (d(k)k^{-1})hk&= h\\
    k^{-1}hk & = h  \\
    k(k^{-1}hk) &=kh \\
    r(k)(hk)& =  kh \\
    (r(h)h)k &=kh \\
    hk&=kh.
  \end{align*}
\end{proof}


It is well known in groups, that given a subgroup $H$, there exists the greatest subgroup of $G$ in which $H$ is normal. Such subgroup is known as normalizer of $H$ and it satisfies some interesting properties. In our case it is natural to ask if is possible to define the normalizer of a subgroupoid. The answer to this question is affirmative as we show below.

\vspace{.3cm}

If $\mathcal{H}$ is a wide  subgroupoid of $\mathcal{G}$, we define the set $$\mathcal{N}_{\mathcal{G}}({\mathcal{H}})=\{g\in \mathcal{G}\mid g^{-1}\mathcal{H}_{r(g)}g=\mathcal{H}_{d(g)}\}$$ which will be called the \textbf{Normalizer of $\mathcal{H}$ in $\mathcal{G}$}. It is clear that $\mathcal{N}_{\mathcal{G}}({\mathcal{H}})\neq \emptyset$ since for $r(g)\in \mathcal{G}_0$ it has $r(g)^{-1}\mathcal{H}_{r(r(g))}r(g)=r(g)\mathcal{H}_{r(g)}r(g)=\mathcal{H}_{r(g)}=\mathcal{H}_{d(r(g))}$ which implies that $r(g)\in \mathcal{N}_{\mathcal{G}}({\mathcal{H}})$ and then $\mathcal{G}_0\subseteq \mathcal{N}_{\mathcal{G}}({\mathcal{H}})$. Note that in the group case this concept coincide with the normalizer of subgroups. The next proposition extends the main properties of the normalizer in groups to the normalizer in groupoids.

\begin{prop}\label{normalizer}
  Let $\mathcal{G}$ be a groupoid and $\mathcal{H}$ a wide subgroupoid of $\mathcal{G}$. Then:
  \begin{enumerate}
    \item $\mathcal{N}_{\mathcal{G}}({\mathcal{H}})$ is a wide subgroupoid of $\mathcal{G}$ that contains $\mathcal{H}$.
    \item $\mathcal{H}$ is a normal subgroupoid of $\mathcal{N}_{\mathcal{G}}({\mathcal{H}})$.
    \item $\mathcal{N}_{\mathcal{G}}({\mathcal{H}})$ is the greatest wide subgroupoid of $\mathcal{G}$ in which $\mathcal{H}$ is normal.
    \item $\mathcal{N}_{\mathcal{G}}({\mathcal{H}})=\mathcal{G}$ if and only if $\mathcal{H}\lhd \mathcal{G}$.
  \end{enumerate}
\end{prop}
\begin{proof}
  1. Note that the width of $\mathcal{H}$ was proved in paragraph previous to the proposition. Now, If $h\in \mathcal{H}$ then $h^{-1}\mathcal{H}_{r(h)}h\subseteq \mathcal{H}\cap \mathcal{G}_{d(h)}=\mathcal{H}_{d(h)}$. If $m\in \mathcal{H}_{d(h)}$ then $r(m)=d(m)=d(h)$ which implies that $\exists hmh^{-1}$ and $hmh^{-1}\in \mathcal{H}_{r(h)}$. Then $\exists h^{-1}(hmh^{-1})h$ and it is clear that $m=h^{-1}(hmh^{-1})h\in h^{-1}\mathcal{H}_{r(h)}h$. Hence $\mathcal{H}_{d(h)}\subseteq h^{-1}\mathcal{H}_{r(h)}h$ and thus $h^{-1}\mathcal{H}_{r(h)}h=\mathcal{H}_{d(h)}$ that is $h\in \mathcal{N}_{\mathcal{G}}(\mathcal{H})$. Then $\mathcal{H}\subseteq \mathcal{N}_{\mathcal{G}}(\mathcal{H})$.

  Let $g,t\in \mathcal{N}_{\mathcal{G}}(\mathcal{H})$ and suppose that $\exists gt$. Then $g^{-1}\mathcal{H}_{r(g)}g=\mathcal{H}_{d(g)}$, $t^{-1}\mathcal{H}_{r(t)}t=\mathcal{H}_{d(t)}$, $d(g)=r(t)$ and thus
  \begin{align*}
     (gt)^{-1}\mathcal{H}_{r(gt)}(gt) & =t^{-1}g^{-1}\mathcal{H}_{r(g)}gt \\
     & =t^{-1}\mathcal{H}_{d(g)}t \\
     & =t^{-1}\mathcal{H}_{r(t)}t \\
     & =\mathcal{H}_{d(t)} \\
     & =\mathcal{H}_{d(gt)}.
  \end{align*}
  If $t\in \mathcal{N}_{\mathcal{G}}(\mathcal{H})$ then $t^{-1}\mathcal{H}_{r(t)}t=\mathcal{H}_{d(t)}$ and thus $t\mathcal{H}_{d(t)}t^{-1}=\mathcal{H}_{r(t)}$. Hence $(t^{-1})^{-1}\mathcal{H}_{r(t^{-1})}t^{-1}=\mathcal{H}_{d(t^{-1})}$ and then $t^{-1}\in \mathcal{N}_{\mathcal{G}}(\mathcal{H})$. Therefore $\mathcal{N}_{\mathcal{G}}(\mathcal{H})$ is a wide subgroupoid of $\mathcal{G}$.

  2. By item 1, $\mathcal{H}\subseteq \mathcal{N}_{\mathcal{G}}(\mathcal{H})$ and since $\mathcal{H}$ is a wide subgroupoid of $\mathcal{G}$ we have that $\mathcal{H}$ is a wide subgroupoid of $\mathcal{N}_{\mathcal{G}}(\mathcal{H})$. Now, consider $m\in \mathcal{N}_{\mathcal{G}}(\mathcal{H})$, $h\in \mathcal{H}$ and suppose that $\exists m^{-1}hm$. Then $r(h)=d(h)=r(m)$, that is $h\in \mathcal{H}_{r(m)}$ and thus $m^{-1}hm\in m^{-1}\mathcal{H}_{r(m)}m=\mathcal{H}_{d(m)}\subseteq \mathcal{H}$. Hence $\mathcal{H}$ is normal in $\mathcal{N}_{\mathcal{G}}(\mathcal{H})$.

  3. Suppose that $\mathcal{T}$ is a wide subgroupoid of $\mathcal{G}$ and that $\mathcal{H}$ is normal in $\mathcal{T}$. If $t\in \mathcal{T}$ then $t^{-1}\mathcal{H}_{r(t)}t=\mathcal{H}_{d(t)}$ and thus $t\in \mathcal{N}_{\mathcal{G}}(\mathcal{H})$. Hence $\mathcal{T}\subseteq \mathcal{N}_{\mathcal{G}}(\mathcal{H})$ and the result follows.

  4. It is evident.
\end{proof}

Normal subgroups are very important in group theory, because they are necessary to construct the quotient group. In the groupoid case, given a wide subgroupoid  $\mathcal{H}$  of $\mathcal{G}$, in \cite{PT} Paques and Tamusiunas define a relation on $\mathcal{G}$ as follows:  for every $g,l \in \mathcal{G}$,

$$ g\equiv_{\mathcal{H}}l \Longleftrightarrow \exists l^{-1}g\,\,\,\,\text{ and} \,\,\,\,l^{-1}g\in \mathcal{H}.$$
 Furthermore, they prove that this relation is a congruence, that is an equivalence relation which is compatible with products. The equivalence class of $\equiv_{\mathcal{H}}$ containing $g$ is the set $g\mathcal{H}=\{gh\mid h\in \mathcal{H} \, \land \, r(h)=d(g)\}$. This set is called left coset of $\mathcal{H}$ in $\mathcal{G}$ containing $g$. Moreover, they prove that if $\mathcal{H}$ is a normal subgroupoid of $\mathcal{G}$ and $\mathcal{G/H}$ is the set of all left cosets of $\mathcal{H}$ in $\mathcal{G}$, then $\mathcal{G/H}$ is a groupoid such that $\exists (g\mathcal{H})(l\mathcal{H})$, if and only if, $\exists gl$ and the partial binary operation is given by $(g\mathcal{H})(l\mathcal{H})=gl\mathcal{H}$ \cite[Lemma 3.12]{PT}. That groupoid is called the \textbf{quotient groupoid} of $\mathcal{G}$ by $\mathcal{H}$.

\vspace{.3cm}

In order to improve the understanding of this work, we finish this section by presenting the notion of groupoid (strong) homomorphism and the first isomorphism theorem. The other isomorphism theorems also remains valid in our context of groupoids. The proofs of these theorems and several examples of application can be consulted in \cite{AMP}.

\begin{defn}
Let $\mathcal{G}$ and $\mathcal{G'}$ be groupoids. A map $\phi: \mathcal{G}\to \mathcal{G'}$ is called groupoid homomorphism if for all $x,y\in \mathcal{G}$, $\exists xy$ implies that $\exists \phi(x)\phi(y)$ and in this case  $\phi(xy)=\phi(x)\phi(y)$. In addition, if $\phi$ is a groupoid homomorphism and for all $x,y\in \mathcal{G}$, $\exists \phi(x)\phi(y)$ implies that $\exists xy$ then $\phi$ is called groupoid strong homomorphism.
\end{defn}

\begin{thm}\label{teoremas de isomorfismos}
(The First Isomorphism Theorem) Let $\phi: \mathcal{G}\to \mathcal{G'}$ be a groupoid strong homomorphism. If $\phi$ is surjective then there exists an strong isomorphism $\overline{\phi}:\mathcal{G}/Ker(\phi)\to \mathcal{G'}$ such that $\phi=\overline{\phi}\circ j$, where $j$ is the canonical homomorphism of $\mathcal{G}$ onto $\mathcal{G}/Ker(\phi)$.
\end{thm}


\section{Center and commutators}

In this section we introduce the center and the commutator subgroupoid and prove several properties of them, which extend  well-known results in groups.

\begin{defn}
  Let $\mathcal{G}$ be a groupoid. We define the center of $\mathcal{G}$ as the set $\mathcal{Z}(\mathcal{G})=\{g\in Iso(\mathcal{G})\mid gh=hg \textnormal{ for all }h\in \mathcal{G}\text{ such that }d(g)=r(h)=d(h)\}$.
\end{defn}

The center of the groupoid $\mathcal{G}$ has analogous properties to those of the groups, as we show in the following
proposition.

\begin{prop}\label{center}
  Let $\mathcal{G}$ be a groupoid and $\mathcal{Z}(\mathcal{G})$ the center of $\mathcal{G}$. Then:
  \begin{enumerate}
    \item $\mathcal{Z}(\mathcal{G})=\bigsqcup_{e\in \mathcal{G}_0}Z(\mathcal{G}_e)$.
    \item $\mathcal{Z}(\mathcal{G})=Iso(\mathcal{G})$ if and only if $\mathcal{G}$ is an abelian groupoid.
    \item $\mathcal{Z}(\mathcal{G})$ is a normal subgrupoid of $Iso(\mathcal{G})$.
    \item If $\mathcal{H}$ is a wide subgrupoid of $\mathcal{Z}(\mathcal{G})$ then it is normal in $Iso(\mathcal{G})$.
  \end{enumerate}
\end{prop}
\begin{proof}
  1. If $g\in \mathcal{Z}(\mathcal{G})$ then $g\in \mathcal{G}_e$ for some $e\in \mathcal{G}_0$ and thus $d(g)=r(g)=e$. If $h\in \mathcal{G}_e$ then $\exists gh$, $\exists hg$ and then $gh=hg$ that is $g\in Z(\mathcal{G}_e)\subseteq \bigsqcup_{e\in \mathcal{G}_0}Z(\mathcal{G}_e)$.

  On the other hand, if $g\in \bigsqcup_{e\in \mathcal{G}_0}Z(\mathcal{G}_e)$ then $g\in Z(\mathcal{G}_e)$ for some $e\in \mathcal{G}_0$. If $h\in \mathcal{G}$ is such that $d(g)=r(h)=d(h)$ then $e=d(h)=r(h)$ and thus $h\in \mathcal{G}_e$ and hence $gh=hg$. That is $g\in \mathcal{Z}(\mathcal{G})$.

  2. It is evident.

  3. First of all that $\mathcal{G}_0\subseteq \mathcal{Z}(\mathcal{G})$. Let $g,h\in \mathcal{Z}(\mathcal{G})$ and suppose that $\exists gh$. Then by item 1. $g\in Z(\mathcal{G}_e)$ and $h\in Z(\mathcal{G}_{e'})$ for some $e,e'\in \mathcal{G}_0$. Then $d(g)=r(g)=e$, $d(h)=r(h)=e'$ and since $d(g)=r(h)$ we have $e=e'$. Thus $h\in Z(\mathcal{G}_e)$ and hence $gh\in Z(\mathcal{G}_e)\subseteq \mathcal{Z}(\mathcal{G})$. If $g\in \mathcal{Z}(\mathcal{G})$ then $g\in Z(\mathcal{G}_e)$ for some $e\in \mathcal{G}_0$ and since $Z(\mathcal{G}_e)$ is a subgroup of $\mathcal{G}_e$ we have $g^{-1}\in Z(\mathcal{G}_e)\subseteq \mathcal{Z}(\mathcal{G})$. Thus $\mathcal{Z}(\mathcal{G})$ is an ample subgroupoid of $\mathcal{G}$.

  Finally, let $g\in Iso(\mathcal{G})$ and $h\in \mathcal{Z}(\mathcal{G})$ such that $r(h)=d(h)=r(g)$. Then $g\in \mathcal{G}_e$ for some $e\in \mathcal{G}_0$ and thus $r(g)=d(g)=e$ and then $r(h)=d(h)=e$. Thus $h\in Z(\mathcal{G}_e)$ and since $\exists g^{-1}hg$ we have that $g^{-1}hg=g^{-1}gh=d(g)h=eh=h\in \mathcal{Z}(\mathcal{G})$. Hence $\mathcal{Z}(\mathcal{G})$ is a normal subgroupoid of $Iso(\mathcal{G})$.

  4. Let $g\in Iso(\mathcal{G})$ and $h\in \mathcal{H}$ such that $r(h)=d(h)=r(g)$. Then $g\in \mathcal{G}_e$ for some $e\in \mathcal{G}_0$ and thus $r(g)=d(g)=e$ and then $r(h)=d(h)=e$. Thus $h\in Z(\mathcal{G}_e)$ and since $\exists g^{-1}hg$ we have that $g^{-1}hg=g^{-1}gh=d(g)h=eh=h\in \mathcal{H}$. Hence $\mathcal{H}$ is a normal subgrupoid of $Iso(\mathcal{G})$.
\end{proof}

If we wish to define the commutator subgroupoid of a groupoid $\mathcal{G}$, then we must start by defining the commutator
of two elements. Thus, if $x,y\in \mathcal{G}$ then $\exists xyx^{-1}y^{-1}$ if and only if $d(x)=r(x)=d(y)=r(y)$ if and only if $x,y\in \mathcal{G}_e$ for some $e\in \mathcal{G}_0$. That is, for $x,y\in \mathcal{G}_e$ for some $e\in \mathcal{G}_0$ we define the commutator of $x,y$ as $[x,y]=x^{-1}y^{-1}xy$.

\begin{defn}
  Let $\mathcal{G}$ be a groupoid. The commutator subgroupoid of $\mathcal{G}$ is given by the set $\mathcal{G}'=\langle [x,y]\mid x,y\in \mathcal{G}_e, e\in \mathcal{G}_0\rangle$.
\end{defn}

Note that $[x,y]^{-1}=(x^{-1}y^{-1}xy)^{-1}=[y,x]$ and $d([x,y])=r([x,y])=e$. Moreover, $xy=yx[x,y]$ and thus $xy=yx$ if and only if $[x,y]=e$. Finally, by following Proposition \ref{SG} the elements of $\mathcal{G}'$ are all the finite products of commutators in $\mathcal{G}$. That is,

$$\mathcal{G}'=\{x_1x_2\cdot\cdot\cdot x_n\mid \exists x_1x_2\cdot\cdot\cdot x_n, n\geq 1 \text{ and each }x_i\text{ is a commutator}\}.$$

More generally, if $\mathcal{H},\mathcal{K}$ are wide subgroupoids of $\mathcal{G}$ then we define $[\mathcal{H},\mathcal{K}]=\langle[x,y]\mid x\in \mathcal{H}_e,y\in \mathcal{K}_e,e\in \mathcal{G}_0\rangle$ and with this notation it is clear that $\mathcal{G}'=[\mathcal{G},\mathcal{G}]$.
The main properties of the commutator subgroupoid are given in the following proposition. Note that all of them are valid in groups.

\begin{prop}\label{commutator}
  Let $\mathcal{G}$ be a groupoid, let $x,y\in \mathcal{G}_e$, $e\in \mathcal{G}_0$, and let $\mathcal{H}$ a wide subgroupoid of $\mathcal{G}$. Then:
  \begin{enumerate}
    \item $\mathcal{G}'=\bigsqcup_{e\in \mathcal{G}_0}\mathcal{G}_e'$.
    \item $\mathcal{G}'=\mathcal{G}_0$ if and only if $\mathcal{G}$ is an abelian groupoid.
    \item If $\mathcal{H}\lhd \mathcal{G}$ then $[\mathcal{H},\mathcal{G}]\leq \mathcal{H}$.
    \item $\mathcal{G}'$ is a normal subgroupoid of $\mathcal{G}$ and $\mathcal{G}/\mathcal{G}'$ is an abelian groupoid.
    \item $\mathcal{G}/\mathcal{G}'$ is the largest abelian quotient of $\mathcal{G}$ in the sense that if $\mathcal{H}\lhd \mathcal{G}$ and $\mathcal{G}/\mathcal{H}$ is abelian, then $\mathcal{G}'\leq \mathcal{H}$. 
    \item If $\sigma :\mathcal{G}\to \mathcal{A}$ is any homomorphism of $\mathcal{G}$ into an abelian groupoid $\mathcal{A}$, then there exists a homomorphism $\theta :\mathcal{G}/\mathcal{G}'\to \mathcal{A}$ such that $\sigma=\theta \circ j$ where $j$ is the canonical homomorphism of $\mathcal{G}$ into $\mathcal{G}/\mathcal{G}'$.
  \end{enumerate}
\end{prop}
\begin{proof}
  1. If $a\in \mathcal{G}'$, then there exist commutators $x_1,x_2,\cdot\cdot\cdot, x_n$ such that $\exists x_1x_2\cdot\cdot\cdot x_n$ and $a=x_1x_2\cdot\cdot\cdot x_n$. Then $x_1,x_2,\cdot\cdot\cdot, x_n\in \mathcal{G}_e$ for some $e\in \mathcal{G}_0$ and thus $a\in \mathcal{G}_e'$. The other inclusion is evident.

  2. Let $x,y\in \mathcal{G}_e$, $e\in \mathcal{G}_0$. Then $\exists x^{-1}y^{-1}xy$ and thus $x^{-1}y^{-1}xy\in \mathcal{G}'$. Then by assumption $x^{-1}y^{-1}xy=e$ which implies that $xy=yx$, that is, $\mathcal{G}_e$ is an abelian group.

  Conversely, if $\mathcal{G}$ is an Abelian groupoid then for $x,y\in \mathcal{G}_e$ it has that $[x,y]=e$ and the result follows.

  3. Since $[\mathcal{H},\mathcal{G}]$ is a subgroupoid of $\mathcal{G}$ it is enough to show that $[\mathcal{H},\mathcal{G}]\subseteq \mathcal{H}$. Then if $[x,y]\in [\mathcal{H},\mathcal{G}]$ then $x\in \mathcal{H}_e$, $y\in \mathcal{G}_e$ for some $e\in \mathcal{G}_0$ and $[x,y]=x^{-1}y^{-1}xy$. Then by assumption $y^{-1}xy\in \mathcal{H}$ and hence $[x,y]\in \mathcal{H}$ which implies that $[\mathcal{H},\mathcal{G}]\subseteq \mathcal{H}$.

  4. By item 1, $\mathcal{G}'$ is a disjoint union of groups, that is, $\mathcal{G}'$ is a groupoid. Moreover, if $e\in \mathcal{G}_0 $ then $e=[e,e]\in \mathcal{G}'$ and thus $\mathcal{G}'$ is a wide subgroupoid of $\mathcal{G}$. Let $g\in \mathcal{G}$ and $a\in \mathcal{G}'$ such that $\exists g^{-1}ag$. Then $r(a)=d(a)=r(g)$ and $a=[x_1,y_1][x_2,y_2]\cdot \cdot \cdot [x_n,y_n]$ where $x_i,y_i\in \mathcal{G}_e$ for some $e\in \mathcal{G}_0$. Then $gg^{-1}=r(g)=r(a)=d(a)=e$ and we obtain
  \begin{align*}
    g^{-1}ag & =g^{-1}[x_1,y_1][x_2,y_2]\cdot \cdot \cdot [x_n,y_n]g \\
     & = g^{-1}[x_1,y_1]e[x_2,y_2]e\cdot \cdot \cdot e[x_n,y_n]g \\
     & = g^{-1}[x_1,y_1]gg^{-1}[x_2,y_2]g\cdot \cdot \cdot g^{-1}[x_n,y_n]g \\
     & =[g^{-1}x_1g,g^{-1}y_1g][g^{-1}x_2g,g^{-1}y_2g]\cdot \cdot \cdot [g^{-1}x_ng,g^{-1}y_ng]\in \mathcal{G}'.
  \end{align*}
  Hence, $\mathcal{G}'$ is a normal subgroupoid of $\mathcal{G}$.

  Now, let $x\mathcal{G}',y\mathcal{G}'\in (\mathcal{G}/\mathcal{G}')_{e\mathcal{G}'}$ for some $e\mathcal{G}'\in (\mathcal{G}/\mathcal{G}')_0$. Then $\exists (x\mathcal{G}')^{-1}(y\mathcal{G}')^{-1}\linebreak (x\mathcal{G}')(y\mathcal{G}')$ and $(x\mathcal{G}')^{-1}(y\mathcal{G}')^{-1}(x\mathcal{G}')(y\mathcal{G}')=(x^{-1}\mathcal{G}')(y^{-1}\mathcal{G}')
  (x\mathcal{G}')(y\mathcal{G}')=\linebreak (x^{-1}y^{-1}xy)\mathcal{G}'=e\mathcal{G}'$. Which implies that $(x\mathcal{G}')(y\mathcal{G}')=(y\mathcal{G}')(x\mathcal{G}')$ that is $(\mathcal{G}/\mathcal{G}')_{e\mathcal{G}'}$ is an abelian group.

  5. Let $[x,y]\in \mathcal{G}'$, then $x,y\in \mathcal{G}_e$ for some $e\in \mathcal{G}_0$. Now, $[x,y]\mathcal{H}=(x^{-1}y^{-1}xy)\mathcal{H}=(x^{-1}\mathcal{H})(y^{-1}\mathcal{H})(x\mathcal{H})(y\mathcal{H})=
  (x^{-1}\mathcal{H})(x\mathcal{H})(y^{-1}\mathcal{H})(y\mathcal{H})=e\mathcal{H}$ and hence $[x,y]\mathcal{H}=e\mathcal{H}$. That is, $[x,y]\in \mathcal{H}$ and then $\mathcal{G}'\subseteq \mathcal{H}$.

  6. First note that if $[x,y]\in \mathcal{G}'$ then $x,y\in \mathcal{G}_e$ for some $e\in \mathcal{G}_0$ and thus $r(x)=d(x)=r(y)=d(y)=e$ which implies that $r(\sigma(x))=d(\sigma(x))=r(\sigma(y))=d(\sigma(y))=\sigma(e)\in \mathcal{A}_{\sigma (e)}$. Then
  \begin{align*}
    \sigma ([x,y]) & =\sigma (x^{-1}y^{-1}xy) \\
     & =\sigma(x^{-1})\sigma(y^{-1})\sigma(x)\sigma(y) \\
     & =\sigma(x^{-1})\sigma(x)\sigma(y^{-1})\sigma(y)\\
     &=\sigma(e)\in \mathcal{A}_0.
  \end{align*}
  Hence $\mathcal{G}'\subseteq Ker (\sigma)$. Know, we define $\theta:\mathcal{G}/\mathcal{G}'\to \mathcal{A}$ by $\theta (x\mathcal{G}')=\sigma (x)$ for each $x\mathcal{G}'\in \mathcal{G}/\mathcal{G}'$. If $x\mathcal{G}'=y\mathcal{G}'$ then $\exists y^{-1}x$ and $y^{-1}x\in \mathcal{G}'\subseteq Ker(\sigma)$. Then $\sigma(y^{-1}x)=\sigma (y)^{-1}\sigma (x)=e'$ for some $e'\in \mathcal{A}_0$ which implies that $\sigma (x)=\sigma (y)$. Hence $\theta$ is a well defined function. Moreover, since $\sigma$ is a homomorphism then $\theta$ is also a homomorphism and it is clear that $\sigma =\theta\circ j$.
  \end{proof}

\section{Inner isomorphisms}

In this section we define the concept of inner isomorphism in groupoids, which coincide naturally with those of groups. We extend several results of these isomorphisms and in particular we prove that the set of all the inner isomorphisms of $\mathcal{G}$ is a normal subgroupoid and it is isomorphic to the quotient groupoid of $\mathcal{G}$ by its center.

If $\mathcal{G}$ is a groupoid, we define $$\mathcal{A}(\mathcal{G})=\{f:\mathcal{G}_e\to \mathcal{G}_{e'}\mid e,e'\in\mathcal{G}_0\textnormal{ and }f\textnormal{ is an isomorphism}\}.$$ Then for $f,g\in \mathcal{A}(\mathcal{G})$ we say that $\exists f\circ g$ if and only if $D(f)=R(g)$ where $D(f)$ and $R(g)$ denotes the domain of $f$ and the range of $g$ respectively and in this case $(f\circ g)(x)=f(g(x))$ for all $x\in D(g)$. With this partial operation the set $\mathcal{A}(\mathcal{G})$ is a groupoid where for $f\in \mathcal{A}(\mathcal{G})$ it has that $d(f)=id_{D(f)}$, $r(f)=id_{R(f)}$ and $f^{-1}$ is the inverse of $f$. The elements of $\mathcal{A}(\mathcal{G})$ will be called partial isomorphisms of $\mathcal{G}$. Note that if $\mathcal{G}$ is a group then the set $\mathcal{A}(\mathcal{G})$ coincide with $Aut(\mathcal{G})$.

In the group case the notion of inner automorphism is very important in several subjects. The next results show the interpretation of this concept in the groupoid case and we extend several results well-known in groups.

\begin{prop}\label{isomorfismos internos}
  Let $\mathcal{G}$ be a groupoid and $g\in \mathcal{G}$. Then the function $\mathcal{I}_g:\mathcal{G}_{d(g)}\to \mathcal{G}_{r(g)}$ defined as $\mathcal{I}_g(x)=gxg^{-1}$ for all $x\in \mathcal{G}_{d(g)}$ is a partial isomorphism of $\mathcal{G}$.
\end{prop}
\begin{proof}
  First note that if $x\in \mathcal{G}_{d(g)}$ then $r(x)=d(x)=d(g)$ which implies that $\exists gxg^{-1}$. Moreover $r(gxg^{-1})=r(g)$ and $d(gxg^{-1})=d(g^{-1})=r(g)$ that is $gxg^{-1}\in \mathcal{G}_{r(g)}$. Then $\mathcal{I}_g$ is a well defined function. If $x,y\in \mathcal{G}_{d(g)}$ then $xy\in \mathcal{G}_{d(g)}$ and then
  \begin{align*}
    \mathcal{I}_g(x) & = gxyg^{-1} \\
     & =gxd(g)yg^{-1} \\
     & =gxg^{-1}gyg^{-1} \\
     & =\mathcal{I}_g(x)\mathcal{I}_g(y).
  \end{align*}
  So $\mathcal{I}_g$ is a homomorphism of groups, in particular it is an strong homomorphism of groupoids.
  Now let $x,y\in \mathcal{G}_{d(g)}$ such that $\mathcal{I}_g(x)=\mathcal{I}_g(y)$ then $g^{-1}xg=g^{-1}yg$ and by the cancellation law (valid for groupoids) we obtain $x=y$ and thus $\mathcal{I}_g$ is an injective function. If $m\in \mathcal{G}_{r(g)}$ then $r(m)=d(m)=r(g)$ which implies that $\exists g^{-1}m$ and $\exists mg$ and thus $\exists g^{-1}mg$. Note that $r(g^{-1}mg)=d(g^{-1}mg)=d(g)$ and thus $g^{-1}mg\in \mathcal{G}_{d(g)}$ and then
  \begin{align*}
    \mathcal{I}_g(g^{-1}mg) & =g(g^{-1}mg)g^{-1} \\
     & =gg^{-1}mgg^{-1}\\
     & =r(g)mr(g) \\
     & =m.
  \end{align*}
  That is $\mathcal{I}_g$ is a surjective function and we conclude that $\mathcal{I}_g$ is a partial strong isomorphism of $\mathcal{G}$. Note in particular that $\mathcal{I}_g$ is an isomorphism of groups.
\end{proof}

The isomorphisms given in Proposition \ref{isomorfismos internos} will be called partial inner isomorphisms of $\mathcal{G}$ and the set of all the inner isomorphisms of $\mathcal{G}$ will be denoted by $\mathcal{I}(\mathcal{G})$. 

We say that the wide subgroupoid $\mathcal{H}$ of $\mathcal{G}$ is \textbf{invariant} by the partial inner isomorphism $\mathcal{I}_g$, $g\in \mathcal{G}$, if $\mathcal{I}_g(\mathcal{H}\cap D(\mathcal{I}_g))=\mathcal{H}\cap R(\mathcal{I}_g)$. That is, if $\mathcal{I}_g(\mathcal{H}_{r(g)})=\mathcal{I}_g(\mathcal{H}\cap \mathcal{G}_{r(g)})=\mathcal{H}\cap \mathcal{G}_{d(g)}=\mathcal{H}_{d(g)}$.


\begin{prop}\label{inner}
  Let $\mathcal{G}$ be a groupoid and $\mathcal{H}$ an wide subgroupoid of $\mathcal{G}$. Then
  \begin{enumerate}
    \item $\mathcal{I}(\mathcal{G})$ is a normal subgroupoid of $\mathcal{A}(\mathcal{G})$.
    \item $\mathcal{I}(Iso(\mathcal{G}))=\{\mathcal{I}e\mid e\in \mathcal{G}_0\}$ if and only if $\mathcal{G}$ is a abelian groupoid.
    \item The function $\Theta :\mathcal{G}\to \mathcal{I}(\mathcal{G})$ defined by $\Theta (g)=\mathcal{I}_g$ for all $g\in \mathcal{G}$ is an strong homomorphism.
    \item The groupoids $\mathcal{G}/\mathcal{Z}(\mathcal{G})$ and $\mathcal{I}(\mathcal{G})$ are isomorphic.
    \item $\mathcal{H}$ is normal if and only if it is invariant for all the partial inner isomorphisms of $\mathcal{G}$.
  \end{enumerate}
\end{prop}
\begin{proof}
  1. If $e\in \mathcal{G}_0$ then the partial inner isomorphism $\mathcal{I}_e:\mathcal{G}_e\to \mathcal{G}_e$ is given by $\mathcal{I}_e(x)=exe^{-1}=exe=x=id_{\mathcal{G}_e}(x)$ for all $x\in \mathcal{G}_e$. That is $\mathcal{I}_e=id_{\mathcal{G}_e}$ which implies that $\mathcal{A}(\mathcal{G})_0\subseteq \mathcal{I}(\mathcal{G})$.

  Let $\mathcal{I}_g, \mathcal{I}_h\in \mathcal{I}(\mathcal{G})$ and suppose that $\exists \mathcal{I}_g\circ \mathcal{I}_h$. Then $D(\mathcal{I}_g)=R(\mathcal{I}_h)$ that is $\mathcal{G}_{d(g)}=\mathcal{G}_{r(h)}$ and thus $d(g)=r(h)$ which implies that $\exists gh$. Now, for $x\in \mathcal{G}_{d(h)}=\mathcal{G}_{d(gh)}$ we obtain $(\mathcal{I}_g\circ \mathcal{I}_h)(x)=\mathcal{I}_g(\mathcal{I}_h(x))=g(hxh^{-1})g^{-1}=(gh)x(gh)^{-1}=\mathcal{I}_{gh}(x)$. That is $\mathcal{I}_g\circ \mathcal{I}_h=\mathcal{I}_{gh}\in \mathcal{I}(\mathcal{G})$. Now, if $\mathcal{I}_g\in \mathcal{I}(\mathcal{G})$ then $\mathcal{I}_{g^{-1}}\in \mathcal{I}(\mathcal{G})$ and $D(\mathcal{I}_g)=\mathcal{G}_{d(g)}=\mathcal{G}_{r(g^{-1})}=R(\mathcal{I}_{g^{-1}})$ that is $\exists \mathcal{I}_g\circ \mathcal{I}_{g^{-1}}$ and $\mathcal{I}_g\circ \mathcal{I}_{g^{-1}}=\mathcal{I}_{r(g)}=id_{\mathcal{G}_{r(g)}}$. Also $D(\mathcal{I}_{g^{-1}})=\mathcal{G}_{d(g^{-1})}=\mathcal{G}_{r(g)}=R(\mathcal{I}_g)$ that is $\exists \mathcal{I}_{g^{-1}}\circ \mathcal{I}_g$ and $\mathcal{I}_{g^{-1}}\circ \mathcal{I}_g=\mathcal{I}_{d(g)}=id_{\mathcal{G}_{d(g)}}$ and hence $(\mathcal{I}_g)^{-1}=\mathcal{I}_{g^{-1}}\in \mathcal{I}(\mathcal{G})$. And thus $\mathcal{I}(\mathcal{G})$ is a wide subgroupoid of $\mathcal{A}(\mathcal{G})$.

  Let $\sigma\in \mathcal{A}(\mathcal{G})$, $\mathcal{I}_g\in \mathcal{I}(\mathcal{G})$ and suppose that $\exists \sigma ^{-1}\circ \mathcal{I}_g\circ \sigma$. Then $R(\mathcal{I}_g)=D(\mathcal{I}_g)=R(\sigma)$ and thus $\mathcal{G}_{r(g)}=\mathcal{G}_{d(g)}=R(\sigma)=D(\sigma ^{-1})$ which implies that $g,g^{-1}\in D(\sigma ^{-1})$. Now for $x\in D(\sigma)$ we have that
  \begin{align*}
    (\sigma ^{-1}\circ \mathcal{I}_g\circ \sigma)(x) & =\sigma ^{-1}(\mathcal{I}_g(\sigma(x))) \\
     & =\sigma ^{-1}(g\sigma (x)g^{-1}) \\
     & =\sigma ^{-1}(g)x\sigma^{-1} (g^{-1}) \\
     & =\sigma ^{-1}(g)x(\sigma^{-1} (g))^{-1} \\
     & =\mathcal{I}_{\sigma^{-1}(g)}(x).
  \end{align*}
Then $\sigma ^{-1}\circ \mathcal{I}_g\circ \sigma=\mathcal{I}_{\sigma^{-1}(g)}\in \mathcal{I}(\mathcal{G})$ and hence $\mathcal{I}(\mathcal{G})$ is a normal subgroupoid of $\mathcal{A}(\mathcal{G})$.

2. Let $g\in \mathcal{G}_e$ for some $e\in \mathcal{G}_0$. Then the partial inner isomorphism $\mathcal{I}_g\in \mathcal{I}(Iso(\mathcal{G}))$ and thus $\mathcal{I}_g=\mathcal{I}_{e'}$ for some $e'\in \mathcal{G}_0$. Then $D(\mathcal{I}_g)=D(\mathcal{I}_{e'})$ and $R(\mathcal{I}_g)=R(\mathcal{I}_{e'})$ and since $r(g)=d(g)=e$ it has that $\mathcal{G}_e=\mathcal{G}_{d(g)}=\mathcal{G}_{r(g)}=\mathcal{G}_{e'}$. Thus $e=e'$ and then $gxg^{-1}=\mathcal{I}_g(x)=\mathcal{I}_e(x)=x$ for all $x\in \mathcal{G}_e$ that is $gx=xg$ for all $x\in \mathcal{G}_e$. Hence $\mathcal{G}_e$ is an abelian group and thus $\mathcal{G}$ is an abelian groupoid.

On the other hand, let $\mathcal{I}_g\in \mathcal{I}(Iso(\mathcal{G}))$, $g\in \mathcal{G}_e$ and $e\in \mathcal{G}_0$. Then $d(g)=r(g)=e$ and then for $x\in \mathcal{G}_e$, $\mathcal{I}_g(x)=gxg^{-1}=gg^{-1}x=x=\mathcal{I}_e(x)$. That is $\mathcal{I}_g=\mathcal{I}_e$ and the result follows.

3. If $g,h\in \mathcal{G}$ and $\exists gh$ then $d(g)=r(h)$. Then $D(\Theta (g))=D(\mathcal{I}_g)=\mathcal{G}_{d(g)}=\mathcal{G}_{r(h)}=R(\mathcal{I}_h)=R(\Theta (h))$ and thus $\exists \Theta (g)\circ \Theta (h)$. Now, if $x\in D(\mathcal{I}_{gh})=\mathcal{G}_{d(gh)}=\mathcal{G}_{d(h)}$ then
\begin{align*}
  \Theta (gh)(x) & =\mathcal{I}_{gh}(x) =(gh)x(gh)^{-1}=(gh)x(h^{-1}g^{-1}) \\
   & =g(hxh^{-1})g^{-1}=\mathcal{I}_g(\mathcal{I}_h(x))=(\mathcal{I}_g\circ \mathcal{I}_h)(x) \\
   &=(\Theta (g)\circ \Theta (h))(x).
\end{align*}
That is $\Theta (gh)=\Theta (g)\circ \Theta (h)$ and then $\Theta$ is a homomorphism of groupoids.

Finally let $g,h\in \mathcal{G}$ and suppose that $\exists \Theta (g)\circ \Theta (h)$. Then $D(\Theta (g))=R(\Theta (h))$, that is $\mathcal{G}_{d(g)}=\mathcal{G}_{r(h)}$ which implies that $d(g)=r(h)$ and thus $\exists gh$. Hence $\Theta$ is an strong homomorphism.

4. First note that the strong homomorphism $\Theta$ of item 3 is surjective.

Now, we prove that $Ker(\Theta)=\mathcal{Z}(\mathcal{G})$. If $g\in Ker(\Theta)$ then $\Theta (g)=\mathcal{I}_g=\mathcal{I}_e$ for some $e\in \mathcal{G}_0$. Thus $\mathcal{G}_{d(g)}=D(\mathcal{I}_g)=D(\mathcal{I}_e)=\mathcal{G}_e$ and $\mathcal{G}_{r(g)}=R(\mathcal{I}_g)=R(\mathcal{I}_e)=\mathcal{G}_e$ which implies that $\mathcal{G}_{d(g)}=\mathcal{G}_{r(g)}=\mathcal{G}_e$ and thus $d(g)=r(g)=e$. Then $g\in \mathcal{G}_e$ and $gxg^{-1}=\mathcal{I}_g(x)=\mathcal{I}_e(x)=x$ for all $x\in \mathcal{G}_e$, that is, $gx=xg$ for all $x\in \mathcal{G}_e$. Hence $g\in Z(\mathcal{G}_e)\subseteq \mathcal{Z}(\mathcal{G})$.

Finally, if $g\in \mathcal{Z}(\mathcal{G})$ then $g\in Z(\mathcal{G}_e)$ for some $e\in \mathcal{G}_0$. If $x\in \mathcal{G}_e$ then $\mathcal{I}_g(x)=gxg^{-1}=xgg^{-1}=x=\mathcal{I}_e(x)$. That is $\Theta (g)=\mathcal{I}_g=\mathcal{I}_e\in \mathcal{I}(\mathcal{G})_0$.

Then by using the first isomomorphism theorem (Theorem \ref{teoremas de isomorfismos}) we obtain that $\mathcal{G}/\mathcal{Z}(\mathcal{G})$ is isomorphic to $\mathcal{I}(\mathcal{G})$.

5. If $g\in \mathcal{G}$ then since $\mathcal{H}$ is normal we have $g^{-1}\mathcal{H}_{r(g)}g=\mathcal{H}_{d(g)}$. That is $\mathcal{I}_g(\mathcal{H}_{r(g)})=\mathcal{H}_{d(g)}$ and the result follows. The converse is analogous.
\end{proof}

\end{document}